\documentclass[12pt]{amsart}
\usepackage{amssymb,amsfonts,latexsym,amscd}

\usepackage[all]{xy}
\usepackage{verbatim}

\newtheorem{theorem}{Theorem}[section]

\newtheorem{proposition}[theorem]{Proposition}
\newtheorem{corollary}[theorem]{Corollary}
\theoremstyle{definition}
\newtheorem{definition}[theorem]{Definition}

\newtheorem{example}[theorem]{Example}

\newtheorem{remark}[theorem]{Remark}

\newcommand{\Ker}{\text{Ker\,}}

\newcommand{\End}{\text{End}}

\newcommand{\FPdim}{\text{\rm FPdim}}

\newcommand{\GL}{\text{GL}}
\newcommand{\Br}{\text{Br}}
\newcommand{\SL}{\text{SL}}
\newcommand{\Hom}{\text{Hom}}

\newcommand{\Id}{{\rm Id}}
\newcommand{\Aut}{{\rm Aut}}

\newcommand{\Rep}{\text{Rep}}
\newcommand{\Vect}{\text{Vec}}

\newcommand{\ot}{\otimes}

\newcommand{\ben}{\begin{enumerate}}
\newcommand{\een}{\end{enumerate}}

\newcommand{\C}{{\mathcal C}}
\newcommand{\D}{{\mathcal D}}

\newcommand{\Z}{{\mathbb Z}}

\begin{document}

\title[Descent and forms of tensor categories]
{Descent and forms of tensor categories}
\author{Pavel Etingof}
\address{Department of Mathematics, Massachusetts Institute of Technology,
Cambridge, MA 02139, USA} \email{etingof@math.mit.edu}

\author{Shlomo Gelaki}
\address{Department of Mathematics, Technion-Israel Institute of
Technology, Haifa 32000, Israel} \email{gelaki@math.technion.ac.il}

\date{\today}


\begin{abstract}
We develop a theory of descent and forms of tensor categories over
arbitrary fields. We describe the general scheme of classification
of such forms using algebraic and homotopical language, and give
examples of explicit classification of forms. We also discuss the
problem of categorification of weak fusion rings, and for the
simplest families of such rings, determine which ones are
categorifiable.
\end{abstract}

\maketitle

\section{Introduction}

The goal of this paper is to give an exposition of the theory of
forms of tensor categories over arbitrary fields, providing a
categorical counterpart of the classical theory of forms of
algebraic structures (such as associative algebras, Lie algebras,
Hopf algebras, algebraic varieties, etc.). We provide a
classification of forms both in the algebraic language and in the
homotopical language, using the theory of higher groupoids,
similarly to \cite{ENO2}. Ideologically, this is not really new,
since it is an application of the general ideology of descent
theory. The point of this paper is to work out the classification of
forms more or less explicitly in the setting of tensor categories,
and to discuss examples of this classification for specific tensor
categories. In particular, we show that even in the simplest cases,
such classification problems reduce to interesting questions in
number theory, such as the classification of constructible regular
polygons (i.e., Fermat primes), the Merkurjev-Suslin theorem, and
global class field theory.

The organization of the paper is as follows. In Section 2, we recall
the classical theory of forms of algebraic structures (i.e., the
theory of descent), and reformulate it in the homotopical language.
In Section 3, we explain a generalization of this theory to
semisimple abelian categories and then to semisimple tensor
categories. Finally, in Section 4, we discuss the theory of
categorification of weak unital based rings by rigid tensor
categories over arbitrary fields.

{\bf Acknowledgments.} We are grateful to V. Ostrik for useful
discussions, in particular for suggesting Theorem \ref{sab}.
The research of the first author was
partially supported by the NSF grant DMS-1000113. The second author
was supported by The Israel Science Foundation (grant No. 317/09).
Both authors were supported by BSF grant No. 2008164.

\section{Forms of algebraic structures}

In this section we give an exposition of
the well known theory of forms of algebraic structures.

\subsection{Definition of a form.}
Let $K$ be a field, and let $L$ be an extension field of $K$. Let
$C$ be an algebraic structure defined over $L$ (associative algebra,
Lie algebra, Hopf algebra, algebraic variety, etc.).

\begin{definition}
A {\it form} of $C$ over $K$ (or a {\it descent} of $C$ to $K$) is
an algebraic structure ${\overline{C}}$ defined over $K$, of the
same type as $C$, together with an $L-$linear isomorphism $\xi:
{\overline{C}}\ot_K L\to C$.

An {\it isomorphism} of two forms $({\overline{C}}_1,\xi_1)$,
$({\overline{C}}_2,\xi_2)$ is an isomorphism $\theta:
{\overline{C}}_1\to {\overline{C}}_2$ not necessarily respecting
$\xi_1,\xi_2$. A {\it framed isomorphism} is an isomorphism $\theta$
such that $\xi_2\circ \theta=\xi_1$.
\end{definition}

\subsection{The first condition.}
Let us recall necessary and sufficient conditions for the existence
of a form of $C$ over $K$. For simplicity let us assume that $L$ is
a \emph{finite Galois extension} of $K$ (the case of infinite Galois
extensions is similar, and the general case can be reduced to the
case of a Galois extension). Let $\Gamma:={\rm Gal}(L/K)$ be the
Galois group of $L$ over $K$.

For any $g\in\Gamma$, let ${}^gC$ denote the algebraic structure
obtained from $C$ by twisting its $L-$structure by means of $g$.
This operation is a functor: any morphism $\beta: C\to D$ gives rise
to a morphism $g(\beta): {}^gC\to {}^gD$, which set-theoretically is
the same as $\beta$ (its only difference from $\beta$ is that its
source and target have been twisted by $g$).

Let us call a pair $(g,\varphi)$ consisting of $g\in \Gamma$ and an
isomorphism \linebreak $\varphi:{}^gC\to C$, a {\em twisted
automorphism} of $C$, and let $\Aut_K(C)$ be the group of all
twisted automorphisms of $C$. We have a group homomorphism $\psi:
\Aut_K(C)\to \Gamma$ whose kernel is the group $\Aut(C)$ of
automorphisms of $C$. Clearly, if a form of $C$ over $K$ exists then
$\psi$ must be surjective, i.e., there is a short exact sequence of
groups
\begin{equation}\label{ext}
1\xrightarrow{} \Aut(C)\xrightarrow{} \Aut_K(C)\xrightarrow{\psi}
\Gamma \xrightarrow{} 1.
\end{equation}
Equivalently, in topological terms, there is a fibration
\begin{equation}
\label{fib} \xymatrix{B\Aut_K(C)\ar[d]^{B\Aut(C)} &&\\
B\Gamma &&}
\end{equation}
where $BG$ denotes the classifying space of a group $G$.

\subsection{The second condition.}
Moreover, if a form of $C$ over $K$ exists,
we must be able to choose a set of representatives
$$
\varphi:=\{\varphi_g\in
\psi^{-1}(g)\mid g\in \Gamma\}
$$
such that the conditions
\begin{equation}\label{1cocyc}
\varphi_{gh}=\varphi_g\circ g(\varphi_h),\,g,h\in \Gamma,
\end{equation}
are satisfied, that is, the extension in (\ref{ext}) must split.
Equivalently, the fibration in (\ref{fib}) must have a section
$\sigma_\varphi$.

\subsection{Sufficiency of the two conditions.}
It turns out that this is also sufficient. Namely, we have
the following standard result.

\begin{proposition}
A form of $C$ over $K$ exists if and only if the extension in
(\ref{ext}) splits, i.e., if and only if the fibration in
(\ref{fib}) has a section. \qed
\end{proposition}

For example, if $C$ is an (associative, Lie, Hopf, etc.) algebra $A$
over $L$ and the extension in (\ref{ext}) splits via $\varphi$ then
the set of fixed points $\overline{A}:=\{a\in A\mid
\varphi_g(a)=a,\, \forall g\in \Gamma\}$ is a form of $A$ over $K$.
Conversely, it is easy to see that any form canonically defines a
splitting.

\begin{example}
Suppose that $\Aut(C)$ is an abelian group. In this case, extension
(\ref{ext}) yields an action of $\Gamma$ on ${\rm Aut}(C)$, and the
obstruction to the existence of a form of $C$ over $K$ lies in
$H^2(\Gamma,\Aut(C))$. Indeed, suppose $\psi$ is surjective, and
pick a collection of pre-images $\varphi_g\in \psi^{-1}(g)$ for
$g\in \Gamma$. Then we have
$$
\varphi_{gh}=\alpha(g,h)\circ \varphi_g\circ g(\varphi_h),
$$
where $\alpha: \Gamma\times \Gamma\to {\rm Aut}(C)$, and it easy to
check that $\alpha$ is a $2-$cocycle: $\alpha\in Z^2(\Gamma,{\rm
Aut}(C))$. Moreover, $\varphi_g$ can be corrected to satisfy
(\ref{1cocyc}) if and only if this cocycle is a coboundary. Thus the
obstruction to the existence of a form is the class $[\alpha]\in
H^2(\Gamma,\Aut(C))$.
\end{example}

\subsection{Classification of forms.}
Let us now recall the classification of forms of $C$ over $K$. For
this purpose, assume that a form exists. Let us fix one such form,
call it ${\overline{C}}$, and classify all the forms of $C$ over $K$
(which in this situation are also called \emph{twisted forms} of
${\overline{C}}$).

As explained above, forms of $C$ over $K$ correspond to
splittings of extension (\ref{ext}) (i.e., homotopy classes
of sections of fibration (\ref{fib})). So, if
we choose such a splitting $\varphi$ (which, in particular,
defines an action of $\Gamma$ on $\Aut(C)$) then for any other splitting
$\varphi'$, we have $\varphi'_g=\lambda_g\circ \varphi_g$, $g\in
\Gamma$, for some function $\lambda:\Gamma\to \Aut(C)$, $g\mapsto
\lambda_g$.

The following proposition is standard.

\begin{proposition}\label{stan}
The following hold:

(i) $\varphi'$ is a splitting if and only if $\lambda\in
Z^1(\Gamma,\Aut(C))$ (i.e., $\lambda$ is a $1-$cocycle).

(ii) $\lambda_1,\lambda_2\in Z^1(\Gamma,\Aut(C))$ determine the same
form up to a framed isomorphism if and only if they define the same
cohomology class in $H^1(\Gamma,\Aut(C))$. \qed
\end{proposition}

\begin{corollary}
The framed isomorphism classes of twisted forms of ${\overline{C}}$
over $K$ are in a natural bijection with the set
$H^1(\Gamma,\Aut(C))$, and the unframed isomorphism classes of
twisted forms are in a natural bijection with orbits of the group
${\rm Aut}(C)^\Gamma={\rm Aut}(\overline{C})$ on  $H^1(\Gamma,{\rm
Aut}(C))$. \qed
\end{corollary}

\begin{remark}
If the group $\Aut(C)$ is abelian, the action
of $\Gamma$ on ${\rm Aut}(C)$ is defined a priori,
and the above discussion shows
that the set of isomorphism (or framed isomorphism) classes of
forms of $C$ over $K$ is naturally a (possibly empty) torsor
$T$ over the group $H^1(\Gamma,\Aut(C))$, which is trivialized
once we choose a form ${\overline{C}}\in T$.
\end{remark}

\section{Forms of categories}

\subsection{Definition of a form.}
Let $K$ be a field (which for simplicity we will assume to be
perfect), and let $L$ be a field extension of $K$.

Let ${\overline{\C}}$ be a semisimple abelian category over $K$ with
finite-dimensional $\Hom-$spaces. Then one can define a semisimple
abelian category $\C:={\overline{\C}}\otimes_K L$, which is the
Karoubian envelope of the category obtained from ${\overline{\C}}$
by extending scalars from $K$ to $L$ in the spaces of morphisms (for
example, if $G$ is a finite group, then $\Rep_K(G)\otimes_K
L=\Rep_L(G)$). Note that under this operation, simple objects of
${\overline{\C}}$ may cease to be simple, and decompose into several
simple pieces (in particular, $\C$ may have more simple objects than
$\overline{\C}$). Also note that this operation respects the
structures of a tensor category, braided tensor category, etc.

Let us say that a semisimple abelian category $\overline{\C}$ over
$K$ is \emph{split} if for any simple object $X\in \overline{\C}$,
one has $\End(X)=K$. If $\overline{\C}$ is split, then extension of
scalars reduces just to tensoring up with $L$, and taking the
Karoubian envelope is not needed.

Let $\C$ be a semisimple $L-$linear category with finite-dimensional
$\Hom-$spaces (possibly with extra structure, e.g., tensor, braided,
etc.). In this section, we develop a theory of forms of such
categories, categorifying the results of the previous section.
\footnote{Although for simplicity we work with semisimple
categories, our constructions can be generalized to the
non-semisimple case, using an appropriate notion of extension of
scalars for linear categories (see e.g. \cite[p. 155]{dm}). Also,
the condition that $K$ is perfect can be dropped without any changes
if we work with absolutely semisimple categories, i.e., categories
in which endomorphism algebras of simple objects are separable
(which means semisimple after any field extension).}

\begin{definition}
A {\it form} of $\C$ over $K$ (or a {\it descent} of $\C$ to $K$) is
a category ${\overline{\C}}$ defined over $K$, of the same type as
$\C$, together with an $L-$linear equivalence $\Xi:
{\overline{\C}}\ot_K L\to \C$.

An {\it equivalence} of two forms $({\overline{\C}}_1,\Xi_1)$,
$({\overline{\C}}_2,\Xi_2)$ of $\C$ over $K$ is an equivalence
$\Theta: {\overline{\C}}_1\to {\overline{\C}}_2$ not necessarily
respecting $\Xi_1,\Xi_2$. A {\it framed equivalence} is an
equivalence $\Theta$ together with an isomorphism of functors
$\Xi_2\circ \Theta\cong \Xi_1$. We will say that an equivalence
admits a framing if it can be upgraded to a framed equivalence.
\end{definition}

The necessary and sufficient conditions for the existence of a form
of $\C$ over $K$ are similar to the case of algebraic structures.
Namely, as before, let us assume for simplicity that $L$ is a finite
Galois extension of $K$, and let $\Gamma:={\rm Gal}(L/K)$ be the
Galois group of $L$ over $K$.
\subsection{The first condition.}
Let $\C$ be an $L-$linear category (abelian, tensor, braided, etc.).
For any $g\in\Gamma$, let ${}^g\C$ denote the category obtained from
$\C$ by twisting its $L-$structure by means of $g$. This operation
is a $2-$functor.

Let us call a pair $(g,\Phi)$ consisting of $g\in \Gamma$ and an
equivalence $\Phi: {}^g\C\to \C$, a {\em twisted auto-equivalence}
of $\C$, and let $\underline{\Aut}_K(\C)$ be the categorical
$(1-)$group (or gr-category) of all twisted auto-equivalences of
$\C$.

\begin{remark}
There is an (equivalent) approach in which one considers an
auto-equivalence $\Phi_g:\C\to \C$, which is \emph{semi-linear
relative to $g$} (see e.g. \cite[p.158]{dm}). A similar approach can
be taken in the algebraic setting of Section 2, as well.

The reason we prefer to use twisted auto-equivalences is that it
exhibits more clearly why the cohomology classes we get as
obstructions or freedoms are with twisted coefficients (i.e., with
coefficients in a non-trivial module).
\end{remark}

We have a categorical group homomorphism $\Psi:
\underline{\Aut}_K(\C)\to \Gamma$ (where $\Gamma$ is the usual
Galois group regarded as a categorical group) whose kernel is the
categorical group $\underline{\Aut}(\C)$ of auto-equivalences of
$\C$. Clearly, if a form of $\C$ over $K$ exists then $\Psi$ must be
surjective, i.e., there is a short exact sequence of categorical
groups
\begin{equation}\label{ext1}
1\xrightarrow{} \underline{\Aut}(\C)\xrightarrow{}
\underline{\Aut}_K(\C)\xrightarrow{\Psi} \Gamma \xrightarrow{} 1.
\end{equation}
Equivalently, in topological terms, there is a fibration
\begin{equation}
\label{fib1} \xymatrix{B\underline{\Aut}_K(\C)
\ar[d]^{B\underline{\Aut}(\C)} &&\\
B\Gamma. &&}
\end{equation}
Here $B\mathcal{G}$ denotes the classifying space of a categorical
group $\mathcal{G}$. Recall that this space is $2-$type, i.e., it
has two non-trivial homotopy groups $\pi_1=G:={\rm Ob}(\mathcal{G})$
and $\pi_2=A:={\rm Aut}({\bold 1})$, and its structure is determined
by an action of $G$ on $A$ and an element of $H^3(G,A)$.

\begin{example}\label{cycgp}
Let $L:=\Bbb C$, let $p$ be a prime of the form $4k-1$, and let
$\C:=\Vect_{\Z/p\Z}^\omega(\Bbb C)$ be the category of
$\Z/p\Z-$graded complex vector spaces with a non-trivial $3-$cocycle
$\omega$. Let $K:=\Bbb R$, and let $g\in {\rm Gal}(\Bbb C/\Bbb R)$
be the complex conjugation. Then
${}^g\C=\Vect_{\Z/p\Z}^{\omega^{-1}}(\Bbb C)$, which is not
equivalent to $\C$, since $-1$ is a quadratic non-residue modulo
$p$. Hence $\C$ is not defined over $\Bbb R$.
\end{example}

\subsection{The second condition.}
Moreover, if a form exists,
we must be able to choose a set of representatives
$$
\Phi:=\{\Phi_g\in
\Psi^{-1}(g)\mid g\in \Gamma\}
$$
such that the conditions
\begin{equation}\label{1cocyc1}
\Phi_{gh}\cong \Phi_g\circ g(\Phi_h),\,g,h\in \Gamma,
\end{equation}
are satisfied. That is, the extension in (\ref{ext1}) must split at
the level of ordinary groups. Equivalently, the fibration in
(\ref{fib1}) must have a section over the $2-$skeleton of the base.

\begin{example}\label{cycgp1}
Let us keep the setting of Example \ref{cycgp}, except that $p$ is a
prime of the form $4k+1$. Let $m\in \Z/p\Z$ be such that $m^2=-1$.
Then we have an equivalence $\Phi_g: {}^g\C\to \C$ such that
$\Phi_g(X_1)=X_m$ (where $X_0:={\bold 1},\dots,X_{p-1}$ are the
simple objects of $\C$). So the square of this functor is not the
identity on simple objects, i.e., equation (\ref{1cocyc1}) is not
satisfied (for any choice of $m$ and $\Phi$). This implies that $\C$
still does not have a real form, even though it is equivalent to its
complex conjugate category.
\end{example}

\subsection{The third condition (vanishing of
$o_\Phi(\C)\in H^3(\Gamma,{\rm Aut}({\rm Id}_\C))$).} Unlike the
case of algebraic structures, the first two conditions are not
sufficient, and there is another obstruction that must vanish in
order for a form to exist. Namely, suppose that exact sequence
(\ref{ext1}) splits at the level of ordinary groups. Then we can
choose functorial isomorphisms
$$
J_{g,h}: \Phi_g\circ g(\Phi_h)\to \Phi_{gh},\,g,h\in \Gamma.
$$
Thus, we obtain two functorial isomorphisms $$\Phi_g\circ
g(\Phi_h)\circ gh(\Phi_k)\to \Phi_{ghk},\,g,h,k\in \Gamma,$$ namely,
$J_{gh,k}\circ J_{g,h}$ and $J_{g,hk}\circ g(J_{h,k})$. If there
exists a form, there should be a choice of $J_{g,h}$ such that these
two functorial isomorphisms are equal. However, we may consider the
element
$$\omega=\omega_\Phi=\{\omega_{g,h,k}\in \Aut(\Phi_{ghk})
\cong \Aut(Id_{\C})\mid g,h,k\in \Gamma\}$$ such that
$$
\omega_{g,h,k}\circ J_{gh,k}\circ J_{g,h}=
J_{g,hk}\circ g(J_{h,k}).
$$

It is easy to see that the splitting of exact sequence (\ref{ext1})
at the level of ordinary groups gives a homomorphism $\Gamma\to {\rm
Aut}_K(\C)$, and hence an action of $\Gamma$ on ${\rm Aut}({\rm
Id}_\C)$. Moreover, it is readily seen that $\omega$ is a
$3-$cocycle for this action: $\omega\in Z^3(\Gamma,{\rm Aut}({\rm
Id}_\C))$. Finally, changing $J_{g,h}$ corresponds to changing
$\omega$ by a coboundary. Thus the last obstruction to existence of
a form of $\C$ over $K$ is the class $o_\Phi(\C)$ of
$\omega=\omega_\Phi$ in $H^3(\Gamma,{\rm Aut}({\rm Id}_\C))$; a form
exists if and only if this obstruction vanishes. Topologically
speaking, this condition is equivalent to the condition that the
section $\sigma_\Phi$ of fibration (\ref{fib1}) over the
$2-$skeleton of $B\Gamma$ defined by $\Phi$ lifts to the
$3-$skeleton. This is equivalent to the condition that $\sigma_\Phi$
lifts to the entire $B\Gamma$, since any section of (\ref{fib1})
over the $3-$skeleton of $B\Gamma$ extends canonically (up to
homotopy) to the entire base (as $B\underline{\rm Aut}(\C)$ has only
two non-trivial homotopy groups).

Summarizing, we obtain the following proposition.

\begin{proposition}
A form of $\C$ over $K$ exists if and only if extension
(\ref{ext1}) splits as an extension of categorical groups, i.e.,
if and only if the fibration in (\ref{fib1}) has a
section.
\end{proposition}

\begin{proof} The ``only if'' part is clear from the above discussion.
To prove the ``if'' part, let $\C$ be a semisimple category over
$L$, and assume that extension (\ref{ext1}) splits via $(\Phi,J)$.
In this case, define a $\Gamma-$stable object of $\C$ to be an
object $X\in \C$ together with an isomorphism $\alpha_g:
\Phi_g(X)\to X$ for each $g\in \Gamma$, such that
$$
\alpha_{gh}\circ J_{g,h}|_X=\alpha_g\circ \Phi_g(\alpha_h).
$$
(Note that simple objects may fail to admit a $\Gamma-$stable
structure since they may be non-trivially permuted by $\Gamma$.)
Given two $\Gamma-$stable objects $X,Y$, the $L-$space $\Hom(X,Y)$
has a natural action of $\Gamma$ defined by $\alpha_g$. Define the
category $\overline{\C}$ of $\Gamma-$stable objects of $\C$ with
$$
\Hom_{\overline{\C}}(X,Y)=\Hom_\C(X,Y)^\Gamma.
$$
Then there is a canonical equivalence $\overline{\C}\otimes_K L\to
\C$, so $\overline{\C}$ is a form of $\C$ over $K$. Indeed, since
$J$ satisfies the $2-$cocycle condition, one can show that any
object $Y\in \C$ is a direct summand in a $\Gamma-$stable object of
$\C$.
\end{proof}

\begin{example} (\emph{Minimal field of definition.})\label{cycgp2}
Keep the setting of Example \ref{cycgp}, with an arbitrary odd prime
$p$. We would like to determine the minimal field of definition of
the category $\C:=\Vect_{\Z/p\Z}^\omega(\mathbb{C})$. Clearly, this
category is defined over the cyclotomic field $L:=\Bbb Q(\zeta)$,
where $\zeta:=e^{2\pi i/p}$. So we would like to find the minimal
subfield $K\subseteq L$ over which this category has a form. By the
main theorem of Galois theory, such subfields $K$ correspond to
subgroups $\Gamma\subseteq \Bbb F_p^\times={\rm Gal}(L/{\Bbb Q})$,
so we are looking for the largest possible subgroup $\Gamma$. Note
that by Examples \ref{cycgp}, \ref{cycgp1}, $K$ cannot be a real
field, and hence $\Gamma$ is of odd order (as it cannot contain
$-1$, i.e., complex conjugation). So if we write $p$ as $p=2^mr+1$,
where $r$ is odd, then the largest $\Gamma$ can be is the group
$\Gamma:=\Bbb Z/r\Bbb Z=(\Bbb F_p^\times)^{2^m}$.

Let us show that this group in fact works, i.e., there is a form of
$\C$ over the corresponding field $K$ (which has degree $2^m$ over
$\Bbb Q$). To see this, note that since $r$ is odd, we have a square
root homomorphism $s: \Gamma\to \Gamma$. Now, for $g\in \Gamma$, let
$\Phi_g:{}^g\C\to \C$ be defined on objects by
$\Phi_g(X_1)=X_{s(g)}$. This can be extended to a tensor functor
(since raising to power $a$ in $\Bbb F_p$ acts on $\omega$ by
$a^{-2}$). Moreover, it is easy to see that the functors $\Phi_g$
satisfy the $1-$cocycle condition, and the obstruction $o_\Phi(\C)$
must vanish since it lies in $H^3(\Gamma,\Z/p\Z)=0$. Thus, there is
a form $\overline{\C}$ of $\C$ over $K$ (which is in fact unique
since $H^2(\Gamma,\Z/p\Z)=0$).
The simple objects of $\overline{\C}$ are $Y_0:={\bold 1}$ and
$Y_1,\dots,Y_{2^m}$ (of Frobenius-Perron dimension $r$).

We see that the initial field of definition $L$ is minimal only very rarely.
This happens if $r=1$ (i.e., the form $\overline{\C}$ is split),
which is equivalent to the condition that
$p$ is a Fermat prime, i.e., a prime of the form $2^{2^n}+1$
(there are only five such primes known, namely, 3, 5, 17, 257, and 65537).

Also, note that if $p=4k-1$ then $m=1$, and thus $\C$ is defined
over the imaginary quadratic field $\Bbb Q(\sqrt{-p})$. The
corresponding form $\overline{\C}$ has three simple objects ${{\bold
1}},X,X^*$, with $\FPdim X=\frac{p-1}{2}=2k-1$, and fusion rules
\begin{equation}\label{sk}
X\otimes X=(k-1)X\oplus kX^*
\end{equation}
and
\begin{equation*}
X\otimes X^*=X^*\otimes X=(2k-1){\bold 1}\oplus (k-1)(X\oplus X^*).
\end{equation*}

Similar analysis applies to the situation when
$\C:=\Vect_{\Z/n\Z}^\omega(\mathbb{C})$, where $\omega$ is a
generator of $H^3(\Z/n\Z,\Bbb C^\times)$, where $n>1$ is a positive
integer, not necessarily a prime. In this case, the a priori field
of definition is $L:=\Bbb Q(e^{2\pi i/n})$, with Galois group
$(\Z/n\Z)^\times$, of order $\varphi(n)$. We write
$\varphi(n)=2^mr$, where $r$ is odd. Then the minimal field of
definition $K$ is of degree $2^m$ over $\Bbb Q$, and the group
$\Gamma={\rm Gal}(L/K)$, of order $r$, is the group of elements of
odd order in $(\Z/n\Z)^\times$. So we have

\begin{proposition}
The minimal field of definition of
$\C:=\Vect_{\Z/n\Z}^\omega(\mathbb{C})$ is $L$ (i.e., $\Gamma=1$) if
and only if the regular $n-$gon can be constructed by compass and
ruler, i.e., if and only if $n$ is a Gauss number, $n=2^sq$, where
$q$ is a product of distinct Fermat primes. \qed
\end{proposition}
\end{example}

\subsection{Classification of forms.}
Let us now describe the classification of forms of $\C$ over $K$ in
the case when they do exist. For this purpose,
fix one such form, call it ${\overline{\C}}$, and
classify all the forms of $\C$ over $K$ (which in this situation are
also called \emph{twisted forms} of ${\overline{\C}}$).

As explained above, forms of $\C$ over $K$ correspond to splittings
of extension (\ref{ext1}) (i.e., homotopy classes of sections of
fibration (\ref{fib1})) together with collections of isomorphisms
$J=(J_{g,h})$. So, if we choose such a splitting $\Phi$ of the
extension of ordinary groups (which, in particular, defines an
action of $\Gamma$ on $\Aut(\C)$) then for any other splitting
$\Phi'$, we have $\Phi'_g=\Lambda_g\circ \Phi_g$, $g\in \Gamma$, for
some $1-$cocycle $\Lambda = \Lambda_{\Phi}: \Gamma\to \Aut(\C)$,
$g\mapsto \Lambda_g$, and conversely, any $1-$cocycle defines a
splitting. Moreover, two $1-$cocycles define equivalent splittings
if and only if they are in the same cohomology class.

Furthermore, it is easy to show that the obstruction $o_{\Phi'}(\C)$
is the pullback of the canonical class $o(\C)\in
H^3(\Aut_K(\C),\Aut({\rm Id}_\C))$ (defining the associativity
isomorphism in $\underline{\Aut}_K(\C)$) under the homomorphism
$\gamma_\Lambda: \Gamma\to \Aut_K(\C)$ defined by $\Lambda$.

Thus, we obtain the following proposition.

\begin{proposition}\label{stan1}
The following hold:

(i) An element $\Lambda\in Z^1(\Gamma,\Aut(\C))$
gives rise to a twisted form of $\overline{\C}$
if and only if $\gamma_\Lambda^*o(\C)=0$; this condition
depends only on the class $[\Lambda]\in H^1(\Gamma,\Aut(\C))$.

(ii) If the condition of (i) is satisfied then framed equivalence
classes of twisted forms of $\overline{\C}$ corresponding to
$\Lambda$ are parameterized by a torsor over
$H^2(\Gamma,\Aut(Id_\C))$.

(iii) Unframed equivalence classes of twisted forms correspond to
orbits of $\Aut(\C)^\Gamma=\Aut(\overline{\C})$ on the data of
(i),(ii).
\end{proposition}

\begin{proof}
For a given $\Phi'$, choices of $J'$ up to equivalence are
parameterized by a torsor over $H^2(\Gamma,\Aut(Id_\C))$.
\end{proof}

\subsection{Forms of semisimple abelian categories.}

\begin{example} Let
$\C:=\Vect(L)$ be the abelian $L-$linear category of
finite-dimensional vector spaces over $L$. Let
$\overline{\C}:=\Vect(K)$. We have ${\rm Aut}(\C)=1$, so there is a
unique choice of $\Phi$. The obstruction $o_\Phi(\C)$ vanishes since
there is a form of $\C$ (namely, $\overline{\C}$). Therefore, the
twisted forms of $\overline{\C}$ are classified by
$H^2(\Gamma,\Aut(\Id_{\C}))=H^2(\Gamma,L^\times)$, which is the
relative Brauer group ${\rm Br}(L/K)$. Indeed, for any $a\in {\rm
Br}(L/K)$, there is a division algebra $D_a$ over $K$ with trivial
center (which splits over $L$), and the form of $\C$ corresponding
to $a$ is just the $K-$linear category of finite-dimensional left
vector spaces over $D_a$.
\end{example}

\begin{example}
Let $\C:=\Vect(L)^n$ be the direct sum of $n$ copies of the category
$\Vect(L)$. Let $\overline{\C}:=\Vect(K)^n$. We have $\Aut(\C)=S_n$,
with a trivial action of $\Gamma$. Thus, choices of $\Phi$
correspond to homomorphisms $\Phi: \Gamma\to S_n$, i.e., commutative
semisimple $K-$algebras $R$ of dimension $n$ with a splitting over
$L$. The obstruction $o_\Phi(\C)$ vanishes for all $\Phi$, since for
any $R$ we have a form $\overline{\C}_R$ of $\C$ over $K$, which is
the category of finite-dimensional $R-$modules. So all the forms
corresponding to $\Phi$ are parameterized by
$H^2(\Gamma,(L^\times)^n)$, where the action of $\Gamma$ corresponds
to $\Phi$.

The algebra $R$ is a direct sum of field extensions of $K$:
$$
R=K_1\oplus\cdots\oplus K_m.
$$
These extensions correspond to orbits of $\Gamma$ on the set
$\{1,...,n\}$. Thus it suffices to consider the case when there is
just one orbit; the general case is obtained by taking the direct
sum. In the case of one orbit, $R$ is a field extension of $K$ of
degree $n$, and $\Gamma$ acts transitively on $(L^\times)^n$. By the
Shapiro Lemma, $H^2(\Gamma,(L^\times)^n)=H^2(\Gamma_1,L^\times)$,
where $\Gamma_1$ is the stabilizer of $1\in \{1,...,n\}$. But
$\Gamma_1$ is the Galois group of $L$ over $R$, so we get that the
twisted forms are parameterized by the relative Brauer group ${\rm
Br}(L/R)$. Indeed, given $a\in {\rm Br}(L/R)$, let $D_a$ be the
corresponding division algebra with center $R$; then the form
corresponding to $R$ is the $K-$linear category of
finite-dimensional left vector spaces over $D_a$.
\end{example}

\subsection{Forms of tensor categories.}
Let us now pass to tensor categories.

\begin{example}\label{ex1}
Let $\overline{K}$ be the algebraic closure of $K$, and let
\linebreak
$\C:=\Vect_{\Z/2\Z}(\overline{K})$ be the tensor category
of $\Z/2\Z-$graded $\overline{K}-$vector spaces. Let
$\overline{\C}:=\Vect_{\Z/2\Z}(K)$. Then $\Aut(\C)=1$, so there is a
unique choice of $\Phi$, the obstruction $o_\Phi(\C)$ vanishes, and
twisted forms are classified by
$H^2(\Gamma,\Aut_{\otimes}(\Id_{\C}))=H^2(\Gamma,\Z/2\Z)$.

Let $\mu_n$ denote the group of roots of unity of order $n$ in
$\overline{K}$, regarded as a $\Gamma-$module. Recall that
$H^2(\Gamma,\mu_n)$ is the group ${\rm Br}_n(K):=\Ker (n|_{{\rm
Br}(K)})$ of $n-$torsion in the Brauer group of $K$.  Indeed, we
have a short exact sequence of $\Gamma-$modules
\begin{equation*}
1 \xrightarrow{} \mu_n\xrightarrow{}
\overline{K}^\times\xrightarrow{n} \overline{K}^\times\xrightarrow{}
1,
\end{equation*}
which yields an exact sequence
$$
H^1(\Gamma,\overline{K}^\times)\to H^2(\Gamma,\mu_n)\to
H^2(\Gamma,\overline{K}^\times)\to H^2(\Gamma,\overline{K}^\times).
$$
Since by Hilbert theorem 90, $H^1(\Gamma,\overline{K}^\times)=0$,
the claim follows.

So the forms of the tensor category $\Vect_{\Z/2\Z}(\overline{K})$
over $K$ are classified by the group $\Br_2(K)$ of elements of order
$\le 2$ in the Brauer group ${\rm Br}(K)$ (which is
$H^2(\Gamma,\Z/2\Z)$).

For example, if $K=\mathbb{R}$ then $\Gamma=\Z/2\Z$ and
$H^2(\Z/2\Z,\Z/2\Z)=\Z/2\Z$, so there is one non-trivial form. Its
simple objects are the unit object $\bf{1}$ and an object $X$
satisfying $X\otimes X=4\cdot \bf{1}$, with $\End(\bf{1})=\mathbb{R}$ and
$\End(X)=\mathbb{H}$ (the algebra of quaternions).

Similar analysis applies to the category
$\Vect_{\Z/2\Z}^{\omega}(\overline{K})$ with a non-trivial
associator (over a field of characteristic different from $2$). Its
forms over $K$ are parameterized by ${\rm Br}_2(K)$, and for $K=\Bbb
R$ there is one trivial and one non-trivial form.
\end{example}

\begin{example} Here is an example where $o_\Phi(\C)\ne 0$.

For a finite group $G$, let ${\rm Out}(G)$ denote the group of outer
automorphisms of $G$, and let $Z(G)$ be the center of $G$. Let $H$
be another finite group, and let $\phi: H\to {\rm Out}(G)$ be a
homomorphism; then $H$ acts naturally on $Z(G)$. It is well known
that there is a canonical Eilenberg-MacLane class $E\in H^3(H,Z(G))$
corresponding to $\phi$, which is not always trivial.

Indeed, here is an example, pointed out to us by David Benson. Take
$G:={\rm SL}(2,\mathbb{F}_9)$ (it is a double cover of the
alternating group $A_6={\rm PSL}(2,\mathbb{F}_9)$) and let
$H:=\mathbb{Z}/2\mathbb{Z}$. Then $Z(G)=\mathbb{Z}/2\mathbb{Z}$ and
${\rm Out}(G)={\rm Out}(A_6)=(\mathbb{Z}/2\mathbb{Z})^2$. Consider
the map $\phi: H\to {\rm Out}(G)$ for which the corresponding
extension of $\mathbb{Z}/2\mathbb{Z}$ by $A_6$ is the Mathieu group
$M_{10}$ (i.e., the image of the non-trivial element of
$\mathbb{Z}/2\mathbb{Z}$ under $\phi$ is the outer automorphism
defined by the composition of conjugation by a matrix whose
determinant is a non-square with the Frobenius automorphism). Then
it is easy to show that there is no extension of
$\mathbb{Z}/2\mathbb{Z}$ by ${\rm SL}(2,\mathbb{F}_9)$ implementing
$\phi$, so the Eilenberg-MacLane class $E\in
H^3(\mathbb{Z}/2\mathbb{Z},\mathbb{Z}/2\mathbb{Z})=
\mathbb{Z}/2\mathbb{Z}$ is non-trivial.

Let $L:=\overline{K}$, let $G$ be a finite group as above, and let
$\C:=\Rep_{\overline{K}}(G)$ be the tensor category of
representations of $G$. In this case, we have a natural homomorphism
$\zeta: {\rm Aut}(G)\to \Aut(\C)$ which factors through ${\rm
Out}(G)$ and lands in $\Aut(\C)^\Gamma$. So any homomorphism $\eta:
\Gamma\to {\rm Out}(G)$ gives rise to a homomorphism $\Gamma\to
\Aut(\C)^\Gamma\subseteq \Aut(\C)$ (which is also a $1-$cocycle),
and hence to a homomorphism $\Phi: \Gamma\to \Aut_K(\C)$. Also, we
have $\Aut_\otimes(\Id_\C)=Z(G)$. Thus, $o_\Phi(\C)\in
H^3(\Gamma,Z(G))$, and one can show that $o_\Phi(\C)=\eta^*(E)$. For
any $G$ and $H\subset {\rm Out}(G)$,
one can find $K$, $L$ and $\eta$
such that the image of $\eta$ is $H$, which gives a desired example.

Note that if $o_\Phi(\C)=0$ then forms obtained from $\Phi$ are
parameterized by a torsor over $H^2(\Gamma,Z(G))$. Also note that if
$\eta$ factors through ${\rm Aut}(G)$ then $o_\Phi(\C)=0$ and
moreover the torsor parameterizing forms is canonically trivial. The
point $0$ of this torsor corresponds to the form of the (algebraic)
group $G$ defined by $\eta$.
\end{example}

\begin{example} Let $\C$ be a split semisimple
tensor category over $L$, and let
$\overline{\C}$ be a form of $\C$ over $K$. Let us say that a
twisted form of $\overline{\C}$ is {\it quasi-trivial} if the
corresponding $1-$cocycle $\Lambda$ is trivial. It follows from the
above that quasi-trivial forms (up to framed equivalence) are
classified by $H^2(\Gamma,\Aut_\otimes(\Id_\C))$. Let us compute
this group in the case $L=\overline{K}$. Let $U_\C$ be the universal
grading group of $\C$ (\cite{GN}), i.e., the group such that $\C$ is
$U_\C-$graded, and any faithful grading of $\C$ comes from a
quotient of $U_\C$ (for example, if $\C:=\Vect_G(L)$, then
$U_\C=G$). Then
$\Aut_\otimes(\Id_\C)=\Hom(U_\C,\overline{K}^\times)$ as a
$\Gamma-$module. So if $(U_\C)_{\rm ab}=\bigoplus_{j=1}^N\Bbb
Z/n_j\Bbb Z$, then by the above discussion we get that quasi-trivial
twisted forms of $\overline{\C}$ are parameterized by
$$
\bigoplus_{j=1}^N {\rm Br}_{n_j}(K).
$$
For instance, for Tambara-Yamagami categories \cite{TY},
$U_\C=\Z/2\Z$, so quasi-trivial forms are parameterized by ${\rm
Br}_2(K)$. This is a generalization of Example \ref{ex1}.
\end{example}

\subsection{Split forms.}
It is an interesting question whether a given split semisimple
tensor category $\C$ over $L$ has a split form over a given subfield
$K\subseteq L$ (see \cite[Section 2.1]{MoS} for a discussion of this
question). In particular, if $G$ is a finite group, it is an
interesting question to determine a minimal number field $K$ over
which there is a split form of the category $\Rep(G)$ of
representations of $G$; e.g., for which $G$ can we take $K=\Bbb Q$?

We note that, as pointed out in \cite[Section 2.1]{MoS}, the
irrationality of characters of $G$ does not imply that $\Rep(G)$ has
no split form over $\Bbb Q$: e.g., if $G$ is abelian then
$\Rep(G)=\Vect_{G^\vee}$ and hence has a split form over $\Bbb Q$.
However, if the action of the group ${\rm
Gal}(\overline{\Bbb Q}/\Bbb Q)$ on irreducible representations of
$G$ does not come from outer automorphisms of $G$, then the minimal
field $K$ must be larger than $\Bbb Q$. This happens for most finite
simple groups of Lie type, since the outer automorphism groups of
these groups are very small, while Galois orbits of representations
increase with the corresponding prime $p$.

\subsection{Forms of braided and symmetric categories.} The theory of
forms of braided and symmetric tensor categories $\C$ is completely
parallel to the theory of forms of usual tensor categories. The only
change is that the categorical group of tensor auto-equivalences of
$\C$ needs to be replaced by the categorical group of braided
(respectively, symmetric) auto-equivalences. This group is a
subgroup of all auto-equivalences at the level of $\pi_1$, and has
the same $\pi_2$ (which is the group of tensor automorphisms of the
identity functor).

As an example consider forms of the nondegenerate braided categories
$\C:=\Vect_{\Z/p\Z}$ where $p$ is an odd prime (see \cite{DGNO}). It
is well known that braidings on this category correspond to
quadratic forms on $\Z/p\Z$. Suppose we are given such a form
$\beta$, which is nondegenerate. Then $\C$ is defined and split over
the cyclotomic field $L:=\Bbb Q(\zeta)$, $\zeta:=e^{2\pi i/p}$, and
one can ask what is the minimal field of definition $K$. Since the
braiding is defined by a quadratic form, the answer is the same as
in Example 3.6: we have to write $p-1$ as $2^mr$, where $r$ is odd,
and $K$ is the fixed field of the subgroup $\Z/r\Z\subseteq {\Bbb
F}_p^\times$ (of degree $2^m$). So we have that $K=L$ if $r=1$,
i.e., if $p$ is a Fermat prime, and $K=\Bbb Q(\sqrt{-p})$ if
$p=4k-1$.

\section{Categorification of weak unital based rings}

\subsection{Definition of weak unital based rings.}
\begin{definition} A {\it weak unital based ring} is a ring $R$
with $\Bbb Z-$basis $b_i$, $i\in I$, containing the unit element
$1$, whose structure constants $N_{ij}^k$ (defined by $b_ib_j=\sum_k
N_{ij}^kb_k$) are non-negative integers, with an involution $*: I\to
I$ defining an anti-involution of $R$, such that the coefficient of
$1$ in $b_ib_j$ is zero if $i\ne j^*$ and positive if $i=j^*$. A
{\it weak fusion ring} is a weak unital based ring of finite rank. A
{\it unital based ring} (respectively, {\it fusion ring}) is a weak
unital based ring (respectively, weak fusion ring) such that the
coefficient of $1$ in $b_ib_{i^*}$ equals $1$.
\end{definition}

It is well known that the Grothendieck ring of a semisimple rigid
tensor category (respectively, fusion category) over an
algebraically closed field $K$ is a unital based ring (respectively,
a fusion ring) (see e.g. \cite{ENO1}). Similarly, we have the
following proposition.

\begin{proposition}
The Grothendieck ring of a semisimple rigid tensor category
(respectively, fusion category) over a general (perfect) field $K$
is a weak unital based ring (respectively, a weak fusion ring).
\end{proposition}

\begin{proof} The properties of a weak unital based ring are obvious,
except for the property of the coefficient of $1$ in $b_ib_j$,
which follows from Schur's lemma.
\end{proof}

This gives rise to the problem of categorification of weak unital
based rings, and in particular weak fusion rings, i.e., finding a
rigid tensor category whose Grothendieck ring is a given weak unital
based ring. This problem is discussed in the following subsection.

\subsection{Categorification of weak fusion rings.}
Let us now discuss the classification problem of categorifications
of given weak fusion rings.

\subsubsection{The rings $R_m$.}\label{rm1}
We start by considering the simplest non-trivial weak fusion rings
$R_m$, with basis $\bold 1$ and $X$, and fusion rules $$X^2=m\bold
1,\,X^*=X,$$ where $m$ is a positive integer.

Recall that the categorifications of $R_1$ over an algebraically
closed field $K$ are $\Vect_{\Z/2\Z}(K)$, and also
$\Vect_{\Z/2\Z}^\omega(K)$ if ${\rm char}(K)\ne 2$ (where $\omega$
is the non-trivial element of $H^3(\Z/2\Z,K^{\times})=\Z/2\Z$).

The classification of categorifications of $R_m$ over any perfect
field $K$ of characteristic $\ne 2$ is given by the following
theorem.

\begin{theorem}\label{prime2}
Let $K$ be a perfect field of characteristic $\ne 2$.

(i) Categorifications $\D_Q^\pm$ of $R_m$ over $K$ are parameterized
by a central division algebra $Q$ over $K$ of dimension $m$ such
that $Q=Q^{op}$, and a choice of sign.

(ii) Categorifications $\D_Q^\pm$ of $R_4$ over $K$ are
parameterized by a choice of a quaternion division algebra $Q$ over
$K$ (i.e. a division algebra $Q_{a,b}$ with generators $x,y$ and
relations $xy=-yx, x^2=a, y^2=b$, for $a,b\in K$), as well as a
choice of sign.

(iii) If $R_m$ admits a categorification over $K$ then $m=4^n$ for
some non-negative integer $n$.

(iv) Any categorification $\D$ of $R_{4^n}$ over $K$ is a
subcategory in a category of the form $\D_{Q_1}^{\pm}\boxtimes
\D_{Q_2}^+\boxtimes\cdots\boxtimes \D_{Q_N}^+$, 
where $Q_i$ are quaternion division algebras.

(v) If $K$ is a number field or $K=\Bbb R$ then $R_m$ is
categorifiable over $K$ if and only if $m=1$ or $m=4$.

(vi) $R_{4^n}$ is categorifiable over $K:=\Bbb C(a_1,...,a_n,b_1,...,b_n)$.
\end{theorem}

\begin{proof}
Let $\overline{\C}$ be a categorification of $R_m$, and let
$\C:=\overline{\C}\otimes_K{\overline{K}}$. Then $\Gamma:={\rm
Gal}(\overline{K}/K)$ acts by automorphisms of the Grothendieck ring
${\rm Gr}(\C)$ as a unital based ring, and $X\in \overline{\C}$
decomposes in $\C$ as $(\bigoplus_{Z\in O}Z)^{\ell}$, where $O$ is
the $\Gamma-$orbit of simple objects corresponding to $X$, and
${\ell}$ is a positive integer. Since $X\otimes X=m\bold 1$, the
orbit $O$ consists of a single element $Z$, and $Z$ is invertible,
with $Z\otimes Z=\bold 1$. So we see that $m={\ell}^2$, and
$\C=\Vect_{\Z/2\Z}(\overline{K})$ or
$\C=\Vect_{\Z/2\Z}^\omega(\overline{K})$. Thus, $\overline{\C}$ is a
form over $K$ of one of these two categories, so it is determined by
a choice of sign \linebreak ($+$ in the first case and $-$ in the
second case) as well as a central division algebra $Q\in {\rm
Br}_2(K)$ over $K$, namely, $Q=\End(X)$. So we see that $\dim
Q={\ell}^2$. This implies (i) and (ii), since in the later case we
have $m=4$, so $Q$ is a quaternion division algebra.

Statement (iii) follows from the following theorem of Brauer.

\begin{theorem}\label{bt} (see \cite{GS}) The dimension of a central
division algebra over $K$ and its order in the Brauer group ${\rm
Br}(K)$ have the same prime factors.
\end{theorem}

Statement (iv) follows from the following  theorem of
Merkurjev.

\begin{theorem} \cite{M}
Any element of order $2$ in ${\rm Br}(K)$ is
represented by a tensor product of quaternion algebras over $K$.
\end{theorem}

Statement (v) follows from a well known result of global class field
theory (see \cite[p.105]{AT68}, \cite[Theorem 3.6]{E}) saying that
any element of order $2$ in $\Br(K)$ for fields $K$ as in (v) is
represented by a quaternion algebra (i.e., taking the tensor product
is not necessary).

Finally, to prove (vi), we
take the category corresponding to the division algebra
$Q:=\otimes_{i=1}^n Q_{a_i,b_i}$ over $K$, which is a division
algebra of dimension $4^n$.
\end{proof}

\subsubsection{The rings $R_{p,r}$.}
Theorem \ref{prime2} can be generalized to the setting involving the
$p-$torsion in the Brauer group for primes $p>2$. Namely, for any
prime $p$ define the weak fusion ring $R_{p,r}$ with basis $\bold 1$
and $X_i$, $i\in \Bbb F_p^\times$, and relations
$$
X_i^*=X_{-i},\ X_iX_j=rX_{i+j}\text{ for }i+j\ne 0,\
X_iX_{-i}=r^2\bold 1.
$$
Thus, $R_{2,r}=R_{r^2}$. Then we have the following result.

\begin{theorem}\label{primep}
Let $K$ be a perfect field which contains a primitive $p-$th root of
unity $\zeta$ (so in particular ${\rm char}K\ne p$).

(i) Categorifications $\D_Q^\omega$ of $R_{p,r}$ over $K$ are
parameterized by a central division algebra $Q$ over $K$ of
dimension $r^2$ such that $Q^{\otimes p}={\rm Mat}_{r^p}(K)$ and a
choice of
 $\omega\in H^3(\Z/p\Z,\overline{K}^\times)=\Z/p\Z$.

(ii) Categorifications $\D_Q^\omega$ of $R_{p,p}$ over $K$ are
parameterized by a choice of a cyclic division algebra $Q$ over $K$
(i.e., an algebra $Q_{a,b,p}$ with generators $x,y$ and relations
$xy=\zeta yx, x^p=a, y^p=b$, for $a,b\in K$), as well as a choice of
$\omega\in H^3(\Z/p\Z,\overline{K}^\times)=\Z/p\Z$.

(iii) If $R_{p,r}$ admits a categorification over $K$ then $r=p^n$
for some non-negative integer $n$.

(iv) Any categorification $\D$ of $R_{p,p^n}$ over $K$ is a
subcategory in a category of the form $\D_{Q_1}^\omega\boxtimes
\D_{Q_2}^1\boxtimes\cdots\boxtimes \D_{Q_N}^1$, where
$Q_i$ are cyclic algebras of dimension $p^2$.

(v) If $K$ is a number field then $R_{p,r}$ is
categorifiable over $K$ if and only if $r=1$ or $r=p$.

(vi) $R_{p,p^n}$ is categorifiable over $K:=\Bbb C(a_1,...,a_n,b_1,...,b_n)$.
\end{theorem}

\begin{proof}
The proof is parallel to the proof of Theorem \ref{prime2}. Let
$\overline{\C}$ be a categorification of $R_{r,p}$, and let
$\C:=\overline{\C}\otimes_K{\overline{K}}$. Then $\Gamma:={\rm
Gal}(\overline{K}/K)$ acts by automorphisms of ${\rm Gr}(\C)$, and
$X:=X_1\in \overline{\C}$ decomposes in $\C$ as $(\bigoplus_{Z\in
O}Z)^{\ell}$, where $O$ is the $\Gamma-$orbit of simple objects
corresponding to $X$, and $\ell$ is a positive integer. Since
$X\otimes X^*=r^2\bold 1$, the orbit $O$ consists of a single
element $Z$, and $Z$ is invertible, with $Z\otimes Z^*=\bold 1$. So
we see that $\ell=r$, and $\C=\Vect_{\Z/p\Z}^\omega(\overline{K})$
for some $\omega\in H^3(\Z/p\Z,\overline{K}^\times)=\Z/p\Z$. Thus,
$\overline{\C}$ is a form of one of these categories over $K$, so it
is determined by a choice of $\omega$ as well as a central division
algebra $Q\in {\rm Br}_p(K)$ over $K$, namely, $Q=\End(X)$. So we
see that $\dim Q=r^2$. This implies (i).

To prove (ii), note that $Q$ is a central division algebra of
dimension $p^2$ over $K$, so by the Albert-Brauer-Hasse-Noether theorem
\cite{BHN}, \cite{AH}, it is a cyclic algebra.

Statement (iii) follows from Brauer's theorem (Theorem \ref{bt}).

Statement (iv) follows from the  theorem of
Merkurjev and Suslin:

\begin{theorem}(\cite{MeS}) Any element of order $p$ in ${\rm
Br}(K)$ is represented by a tensor product of cyclic division
algebras of dimension $p^2$ over $K$.
\end{theorem}

Statement (v) follows from a well known result of global class field
theory (see \cite[p.105]{AT68}, \cite[Theorem 3.6]{E}) saying that
any element of order $p$ in $\Br(K)$ for fields $K$ as in (v) is
represented by a cyclic algebra of dimension $p^2$ (i.e., taking the
tensor product is not necessary).

Finally, to prove (vi), we take the category corresponding to the
division algebra $Q:=\otimes_{i=1}^n Q_{a_i,b_i,p}$ over $K$, which
is a division algebra of dimension $p^{2n}$.
\end{proof}

\subsubsection{The rings $S_k$}\label{sk1}
For a positive integer $k$ define the weak fusion ring $S_k$
with basis $\bold 1,X$ and relations
$$
X^2=k\bold 1+(k-1)X,\, X^*=X.
$$
Let us classify categorifications $\overline{\C}$ of this ring over
a field $K$, say, of characteristic zero. We have $\FPdim(X)=k$.
Passing to the algebraic closure, we get $X=(\bigoplus_{i\in
O}X_i)^{\ell}$, where $O$ is a Galois group orbit. Now, it is clear
that $\FPdim(X_i)=1$ (otherwise we would not be able to write a
decomposition for $X_i\otimes X_i^*$), so we get $k=\ell |O|$. Also,
from the equation for $XX^*=X^2$ we get $\ell^2|O|=k$, which implies
that $\ell=1$, $|O|=k$. So $\C:=\overline{\C}\otimes_K \overline{K}$
is the category $\Vect_G^\omega(\overline{K})$ for some group $G$
and $3-$cocycle $\omega\in H^3(G,\overline{K}^\times)$.

The absolute Galois group $\Gamma$ of $K$ acts on $G$ and has two
orbits, namely $\lbrace{1\rbrace}$ and $O$. This means that all
elements of $G$ other than $1$ have the same order, which then has
to be a prime $p$, such that $k+1=p^n$ for some positive integer
$n$. Moreover, if $G$ were non-abelian then the Galois group would
have to preserve its non-trivial proper subgroup (the center), which
is impossible because we have just two orbits. Thus $G$ must be
abelian, hence a vector space $V$ over $\Bbb F_p$, and we have a
homomorphism $\Phi: \Gamma\to \GL(V)$ such that $\Gamma$ acts
transitively on non-zero vectors, i.e., the image
$\overline{\Gamma}$ of $\Gamma$ in $\GL(V)$ is a transitive finite
linear group.

On the other hand, if we have any surjective homomorphism \linebreak
$\Phi: \Gamma\to \GL(V)$ then for $\omega=1$ we obtain a
categorification of $S_k$ over an appropriate field $K$ (the
category of representations of the twisted form of the algebraic
group $V^*$ defined by the homomorphism $\Phi$). Thus, we obtain the
following proposition, which is somewhat similar to \cite[Corollary
7.4]{EGO}.

\begin{proposition}
The weak fusion ring $S_k$ is categorifiable over a suitable field
if and only if $k+1$ is a prime power. \qed
\end{proposition}

For example, for $k\le 10$, $S_k$ is categorifiable except for $k=5,9$.

The problem of explicit classification of categorifications of $S_k$
for $k=p^n-1$ is rather tricky. For simplicity consider the case
when $p>3$, and $K$ contains a primitive root of unity of order $p$.
In this case, $H^3(V,\overline{K}^\times)=S^2V^*\oplus \wedge^3V^*$
as a $\Gamma-$module. Clearly, $\omega$ must be $\Gamma-$invariant.
Since $\overline{\Gamma}$ is transitive, $\omega$ cannot have a
non-zero component $q_\omega$ in $S^2V^*$, since the level sets of
$q_\omega(v,v)$ are invariant under $\Gamma$. So we have $\omega\in
(\wedge^3V^*)^\Gamma$, and each such $\omega$ will give rise to a
categorification. Apart from the case of vanishing $\omega$, if
$m=3d$, we have the example $\Gamma=\SL_3(\Bbb F_q)$, where $q=p^d$.
In this case, we can take $\omega(v,w,u)=\psi(v\wedge w\wedge u)$,
where $v,w,u\in \Bbb F_q^3$, and $\psi: \Bbb F_q\to \Bbb F_p$ is any
non-zero linear function. The full classification of
categorifications can be obtained by using the known classification
of transitive finite linear groups (see e.g. \cite{BH}); we will not
do it here.

\subsubsection{The rings $T_k$.}
Now, consider the fusion ring $T_k$ defined by (\ref{sk}), where
$k$ is a non-negative integer. Let us classify categorifications
$\overline{\C}$ of this ring over a field $K$, say, of
characteristic zero. We have $\FPdim(X)=2k-1$. Passing to the
algebraic closure, we get $X=(\bigoplus_{i\in O}X_i)^{\ell}$, where
$O$ is a Galois group orbit. Now, it is clear that $\FPdim(X_i)=1$
(otherwise we would not be able to write a decomposition for
$X_i\otimes X_i^*$), so we get $2k-1=\ell |O|$. Also, from the
equation for $X\otimes X^*$ we get $\ell^2|O|=2k-1$, which implies
that $\ell=1$, $|O|=2k-1$. So $\C:=\overline{\C}\otimes_K
\overline{K}$ is the category $\Vect_G^\omega(\overline{K})$ for
some group $G$ and $3-$cocycle $\omega\in
H^3(G,\overline{K}^\times)$.

The absolute Galois group $\Gamma$ of $K$ acts on $G$ and has three
orbits, namely $\lbrace{1\rbrace}$, $O$, and $O^{-1}$. This means
that all elements of $G$ other than $1$ have the same order, which
then has to be an (odd) prime $p$, such that $4k-1=p^{2n+1}$ for
some non-negative integer $n$. Moreover, if $G$ were non-abelian
then the Galois group would have to preserve its non-trivial proper
subgroup (the center), which is impossible because we have just
three orbits. Thus $G$ must be abelian, hence a vector space $V$ of
dimension $2n+1$ over $\Bbb F_p$, and we have a homomorphism $\Phi:
\Gamma\to \GL(V)$.

Now, for any such $p,n$, we can take $V=\Bbb F_q$, where
$q=p^{2n+1}$, and $\Gamma=(\Bbb F_q^\times)^2$. Since $p=4l-1$, and
$q$ is an odd power of $p$, we get that $-1$ is a non-square in
$\Bbb F_q$, so $\Gamma$ satisfies the above condition, and we get a
categorification with trivial $\omega$. Thus, we obtain the
following proposition.

\begin{proposition}
The weak fusion ring $T_k$ is categorifiable over a suitable field
if and only if $4k-1$ is an (odd) power of a prime. \qed
\end{proposition}

For example, for $k\le 10$, $S_k$ is categorifiable except for $k=4,9,10$.

\subsubsection{Weak fusion rings of rank $2$}
Here is a generalization of the results of Subsection \ref{sk1} to
general weak fusion rings of rank $2$, proposed by V. Ostrik (based
on an argument by Josiah Thornton).

For a positive integer $a$
and a nonnegative integer $b$ define the weak fusion ring $S_{a,b}$
with basis $\bold 1,X$ and relations
$$
X^2=a\bold 1+bX,\, X^*=X.
$$

\begin{theorem}\label{sab}
The weak fusion ring $S_{a,b}$
is categorifiable over a suitable field of characteristic zero
if and only if either $a=b=1$ or there is a prime $p$ and
integers $m\ge 0,n\ge 1$ such that $a=p^{2m}(p^n-1)$ and $b=p^m(p^n-2)$.
\end{theorem}

Note that in the case $b=0$ (i.e., $p=2$, $n=1$) Theorem \ref{sab}
reduces to the result of Subsection \ref{rm1} (in characteristic
$0$), and in the case $a=b+1$ (i.e., $m=0$) Theorem \ref{sab}
reduces to the result of Subsection \ref{sk1}.

\begin{proof}
Let $\overline{\C}$ be a categorification  of $S_{a,b}$ over
a field $K$ of characteristic zero.
Passing to the algebraic closure, we get $X=(\bigoplus_{i\in
O}X_i)^{\ell}$, where $O$ is a Galois group orbit. Consider two cases.

{\bf Case 1.} $|O|=1$. Then $\C:=\overline{\C}\otimes_K
\overline{K}$ has two simple objects, $\bold 1$ and $Y$, and
$Y\otimes Y=\bold 1\oplus rY$. By a theorem of Ostrik \cite{O},
$r=0$ or $r=1$. If $r=0$, and $X=\ell Y$, then $b=0$ and
$a={\ell}^2$, so the Theorem follows from Theorem \ref{prime2}
(namely, $p=2$, $n=1$, and $m$ is arbitrary).

If $r=1$, then $\C$ is the Yang-Lee category or its Galois
conjugate. Let us show that any form $\overline{\C}$ of this
category is split.
\footnote{Here is another proof of this fact. Suppose we have a
non-split form with simple objects $\bold 1,X$ and ${\rm End}(X)=D$,
a central division algebra of dimension $s^2$ over a ground field
$K$ (where $s>1$). Then extension of scalars turns $X$ into $sY$, so
the fusion rule for this form is $X^2=s^2\bold 1+sX$, and $X^*=X$.
Thus we have $D^{op}\cong D$, and $D\otimes D$ (which sits inside
$\End(X\otimes X)$) injects into $\End(s^2\bold 1\oplus sY)={\rm
Mat}_{s^2}(K)\oplus {\rm Mat}_s(D)$. But since $D^{op}\cong D$,
$D\otimes D\cong {\rm Mat}_{s^2}(K)$, and there is no homomorphisms
from ${\rm Mat}_{s^2}(K)$ to ${\rm Mat}_s(D)$ (since the former
contains nilpotent elements of nilpotency degree $>s$). This is a
contradiction. Thus, $\overline{\C}$ is split and we have $a=b=1$.}
First of all, any field of definition of a
Yang-Lee category must contain $\sqrt{5}$, since it occurs in the
M\"uger's squared norm of the nontrivial object $Y$. Next, it is
explained in \cite{O} that the Yang-Lee category is defined as a
split category over $\Bbb Q(\sqrt{5})$. Finally, this category has
no nontrivial tensor auto-equivalences or tensor automorphisms of
the identity functor, so by the theory of forms of tensor
categories, the split form has no nontrivial twists, as desired.

{\bf Case 2.} $|O|>1$. In this case we use the argument due to
J. Thornton \cite{Th}. Namely, we claim that $\FPdim(X_i)=1$.
To see this, note first that since $X_i$ are permuted by the
Galois group, $\FPdim(X_i)=d$ is independent of $i$.
Next, if $i\ne j$ (which is possible since $|O|>1$)
then $X_i\otimes X_j^*$ is a direct sum of $X_k$, so
$d$ is an integer. Finally, $X_i\otimes X_i^*$
is $\bold 1$ plus a sum of $X_k$, so
$d^2-1$ is divisible by $d$, hence $d=1$.

So ${\rm FPdim}(X)=\ell |O|$, and we get
$$
(\ell |O|)^2=a+b\ell |O|.
$$
Also, from the relation $XX^*=X^2$ we get $a=\ell^2|O|$, which gives
$ b=\ell (|O|-1)$.

Now, consider the group $V$ formed by the objects $X_i$ and the
unit object $X_0=\bold 1$. Since $V$ consists of two Galois orbits,
it is a vector space over $\Bbb F_p$ for some prime $p$.
Thus, $|O|=p^n-1$ for some $n\ge 1$.

It remains to determine $\ell$. To this end, let $\Gamma$ be the
Galois group of $\overline{K}$ over $K$ (which acts on $V$), and let
$\Gamma'$ be the stabilizer of some $0\ne i\in V$. Let
$L:=\overline{K}^{\Gamma'}$. Then $L$ is a finite extension of $K$,
and $\overline{\C}\otimes_K L$ has simple objects $Y_i$ labeled by
$i\in V$ (with $Y_0=\bold 1$ and $Y_i^*=Y_{-i}$), and one has
$Y_i\otimes Y_j=\ell Y_{i+j}$ if $i\ne -j$, $Y_i\otimes
Y_{-i}=\ell^2\bold 1$. Let $D_i:=\End(Y_i)$. Then $D_i$ has
dimension $\ell^2$ and it defines an element of order $p$ in the
Brauer group ${\rm Br}(L)$. So by Brauer's theorem (Theorem
\ref{bt}), $D_i$ has dimension $p^{2m}$ for some nonnegative integer
$m$. This implies that $\ell=p^m$ for some $m\ge 0$, and thus $a,b$
are as required.

Conversely, if $a=p^{2m}(p^n-1)$ and $b=p^m(p^n-2)$, then the ring
$S_{a,b}$ admits a categorification. In showing this, we may (and
will) assume that $m>0$, since the case $m=0$ is considered in
Subsection \ref{sk1}. Namely, take the category $\C:=\Vect_{{\Bbb
F}_q}$, where $q=p^n$. Assume that the field $K$ is such that the
Galois group $\Gamma={\rm Gal}(\overline{K}/K)$ has a normal
subgroup $\Gamma'$ with $\Gamma/\Gamma'=\Bbb F_q^\times$, and let
$L\subseteq \overline{K}$ be the fixed field of this subgroup (it is
a cyclic Galois extension of $K$ of degree $p^n-1$). So we can make
$\Gamma$ act on $\Bbb F_q$ by multiplications so that $\Gamma'$ acts
trivially. We have a form $\overline{\C}'$ of $\C$ over $K$ defined
by this action, which is split over $L$ (namely, the category of
representations of the twisted form of the additive group of $\Bbb
F_q$ corresponding to the action of $\Gamma$ on $\Bbb F_q$); this
form categorifies the ring $S_k$ with $k=p^n-1$. Now, this form can
be twisted by an element $\eta$ of
$$
H^2(\Gamma,\Hom(\Bbb F_q,\mu_p))= H^2(\Gamma',\Hom(\Bbb
F_q,\mu_p))^{\Gamma/\Gamma'}= \Hom(\Bbb F_q,{\rm Br}_p(L))^{\Bbb
F_q^\times},
$$
where the action of $\Bbb F_q^\times$ on
${\rm Br}_p(L)$ is through the isomorphism of
$\Bbb F_q^\times$ with $\Gamma/\Gamma'$.

Now we would like to choose a suitable field $K$. We will take
$L:={\Bbb C}(a_{ijk},b_{ijk})$, where $i=1,\dots,m$,
$j=1,\dots,p^n-1$ and $k=1,\dots,n$. Then we can make the group
$\Bbb F_q^\times$ act on $L$ by cyclic permutations of $j$, and
define $K$ to be the subfield of invariants. Then $K$, $L$ have the
required properties. Define $D_{jk}$ to be the division algebra over
$L$ given by the formula
$$
D_{jk}:=\otimes_{i=1}^n D_{ijk},
$$
where $D_{ijk}$ is generated by $x_{ijk}$ and $y_{ijk}$
with
$$
x_{ijk}^p=a_{ijk},\, y_{ijk}^p=b_{ijk},\, x_{ijk}y_{ijk}=\zeta
y_{ijk}x_{ijk},
$$
where $\zeta$ is a primitive $p-$th root of unity. Let $E$ be the
subgroup in ${\rm Br}_p(L)$ generated by $D_{jk}$,
$j=1,\dots,p^n-1$, $k=1,\dots,n$. Since $D_{jk}$ are linearly
independent vectors in ${\rm Br}_p(L)$, the space $E$ is isomorphic
to the space of matrices over $\Bbb Z/p\Bbb Z$, of size $n$ by
$p^n-1$, with $\Bbb F_q^\times=\Bbb Z/(p^n-1)\Bbb Z$ acting by
cyclic permutations of columns. Thus,
$$
\Hom(\Bbb F_q,E)^{\Bbb F_q^\times}=\Hom(\Bbb F_q,(\Bbb Z/p\Bbb Z)^n)
\subseteq \Hom(\Bbb F_q,{\rm Br}_p(L))^{\Bbb F_q^\times}.
$$
Take a nondegenerate element $\eta$ from this group
(i.e., defining an isomorphism $\Bbb F_q\to (\Bbb Z/p\Bbb Z)^n$).
Then the twist $\overline{\C}$ of $\overline{\C}'$ by $\eta$
is a categorification of $S_{a,b}$, as desired.
\end{proof}

\begin{remark}
The assumption of characteristic zero is needed here because
Ostrik's classification \cite{O} of fusion categories of rank $2$ is
unknown in positive characteristic. All the other arguments in this
subsection can be extended to positive characteristic.
\end{remark}

\end{document}